\documentclass[11pt]{article}

\usepackage[margin=1in]{geometry}
\usepackage{iftex}
\ifPDFTeX
  \usepackage[T1]{fontenc}
  \usepackage[utf8]{inputenc}
  \usepackage{lmodern}
\else
  \usepackage{fontspec}
\fi
\usepackage{amsmath,amssymb,amsthm,mathtools}
\usepackage{microtype}
\usepackage[colorlinks=true,citecolor=blue,linkcolor=blue,urlcolor=blue]{hyperref}
\hypersetup{pdftitle={A Higher-Order Clique Density Theorem}, pdfauthor={Heng Li, Hong Liu, Yixiao Zhang}, pdfkeywords={clique density theorem, graphons, Lovasz-Simonovits profile, graphon stability, edit-distance stability}}

\newtheorem{theorem}{Theorem}[section]
\newtheorem{lemma}[theorem]{Lemma}
\newtheorem{proposition}[theorem]{Proposition}
\newtheorem{definition}[theorem]{Definition}

\newcommand{\eps}{\varepsilon}
\newcommand{\bino}[2]{\binom{#1}{#2}}
\newcommand{\cut}{\square}

\title{A Higher-Order Clique Density Theorem}
\author{Heng Li\thanks{School of Mathematics, Shandong University, Jinan, China, and Extremal Combinatorics and Probability Group (ECOPRO), Institute for Basic Science (IBS), Daejeon, South Korea. Supported by the National Natural Science Foundation of China (12501487), by China Scholarship Council and the Institute for Basic Science (IBS-R029-C4). Email: \texttt{heng.li@sdu.edu.cn}.}
\and Hong Liu\thanks{Extremal Combinatorics and Probability Group (ECOPRO), Institute for Basic Science (IBS), Daejeon, South Korea. Supported by the Institute for Basic Science (IBS-R029-C4). Email: \texttt{hongliu@ibs.re.kr}.}
\and Yixiao Zhang\thanks{Center for Discrete Mathematics, Fuzhou University, Fuzhou, China. Supported by the National Key R\&D Program of China (No.~2023YFA1010202) and the Doctoral Student Program of the Young S\&T Talents Cultivation Project, CAST. Email: \texttt{fzuzyx@gmail.com}.}}
\date{\today}

\begin{document}
\maketitle

\begin{abstract}
Reiher's clique density theorem determines the sharp lower envelope for the density of $K_r$ at fixed edge density. 
We prove a higher-order version in which the prescribed quantity is itself a clique density. 
For every $3\le s<r$, we determine the minimum possible $K_r$-density among graphons with prescribed $K_s$-density. 
For $s\ge3$ the constraint is genuinely nonlinear and leaves the edge density undetermined; nevertheless, on the positive range the sharp lower boundary is the classical multipartite edge-to-clique profile, reparametrised by $K_s$-density.

We also prove stability on the positive branches of this profile: at every interior point, near extremality forces cut-distance closeness to the corresponding extremal family at the induced edge density.
\end{abstract}

\section{Introduction}

Tur\'an's theorem~\cite{Turan1941} identifies the edge-density threshold below which the complete graph $K_r$ can be avoided. 
A fundamental quantitative refinement asks not merely when $K_r$ must appear, but how many copies are forced once this threshold is crossed. 
In its classical form, one fixes the edge density and minimises the density of $r$-cliques.

This is the Erd\H{o}s--Rademacher clique-density problem, a long-standing problem in extremal graph theory whose study goes back to the late 1950s and early 1960s. 
For triangles, early estimates are due to Goodman~\cite{Goodman1959} and Nordhaus--Stewart~\cite{NordhausStewart1963}; Moon and Moser~\cite{MoonMoser1962} treated four-cliques and established a general adjacent-clique inequality. 
The problem continued to drive a line of inequalities for complete subgraphs, including work of Bollob\'as~\cite{Bollobas1976}, Had\v ziivanov and Nikiforov~\cite{HadziivanovNikiforov1978}, and Fisher~\cite{Fisher1989}. 
Lov\'asz and Simonovits~\cite{LovaszSimonovits1983} conjectured the exact asymptotic lower envelope and the extremal examples: complete multipartite graphs in which all but at most one part have the same size. 
After the triangle case was proved by Razborov~\cite{Razborov2007,Razborov2008} and several further cases were treated by Nikiforov~\cite{Nikiforov2011}, Reiher~\cite{Reiher2016} finally proved the Lov\'asz--Simonovits conjecture in full.

We study the same lower-bound problem with the edge-density coordinate replaced by a clique-density coordinate. 
Given $2\le s<r$, we determine the minimum possible $K_r$-density among graphons with prescribed $K_s$-density. 
When $s=2$ this is precisely Reiher's theorem.  
For $s\ge3$, however, the constraint is already nonlinear and does not simply prescribe the average degree. 

A priori this leaves room for new extremal behaviour: the required $K_s$'s might be distributed very unevenly, rather than spread according to the multipartite construction. Our result shows that this gives no advantage: the sharp lower boundary is still the same multipartite curve. The main new difficulty is to show that extremality under the $K_s$-constraint nevertheless recovers the edge-density behaviour of the multipartite construction.

We work in the language of dense graph limits, introduced by Lov\'asz and Szegedy~\cite{LovaszSzegedy2006}; see Lov\'asz~\cite{Lovasz2012} for background. 
The related graph-profile viewpoint appears, for instance, in work of Hatami and Norine~\cite{HatamiNorine2011}, Blekherman, Raymond, Singh and Thomas~\cite{BlekhermanRaymondSinghThomas2022}, and Blekherman and Raymond~\cite{BlekhermanRaymond2026}. Throughout, a graphon is a symmetric measurable function $W:\Omega^2\to[0,1]$ on a standard atomless probability space $(\Omega,\mu)$, considered up to measure-preserving equivalence. 
We write $\|Z\|_\cut:=\sup_{A,B}|\int_{A\times B}Z\,d\mu^2|$ for the cut norm, where $A,B$ range over measurable sets. 
The \emph{cut distance} $\delta_\cut(U,W)$ is obtained by taking the infimum of
$\|U-W^\varphi\|_\cut$ over all measure-preserving relabellings $\varphi$, where
$W^\varphi(x,y):=W(\varphi(x),\varphi(y))$.
For $j\ge1$, the $K_j$-density in $W$ is defined as
\[
 g_j(W):=\frac1{j!}\int_{\Omega^j}\prod_{1\le a<b\le j} W(x_a,x_b)\,d\mu(x_1)\cdots d\mu(x_j).
\]
For a finite graph $G$, let $N_j(G)$ be the number of vertex subsets spanning $K_j$ and put $g_j(G):=N_j(G)/|V(G)|^j$. 
With the usual step graphon $W_G$, this normalisation gives $g_j(G)=g_j(W_G)$.

We now describe the candidate profile. 
In a complete multipartite graphon, a clique can use at most one vertex from each part, so clique densities are elementary symmetric functions of the part sizes. 
For integers $u\ge j$, put $\sigma_j(u):=\bino{u}{j}u^{-j}$, the $K_j$-density of the balanced complete $u$-partite graphon. 
For $0\le a\le1/(u-1)$, put
\[
 P_j(u,a):=\bino{u}{j}u^{-j}(1+a)^{j-1}\bigl(1-(j-1)a\bigr).
\]
This is the $K_j$-density of the complete $u$-partite graphon with $u-1$ parts of measure $(1+a)/u$ and one part of measure $(1-(u-1)a)/u$. 
As $a$ increases, the exceptional part shrinks to zero, so this branch connects the balanced $u$-partite and $(u-1)$-partite graphons.

\begin{definition}[The profile $\Phi_{s,r}$]
Let $2\le s<r$. 
Define $\Phi_{s,r}:[0,\frac1{s!}]\to[0,\frac1{r!}]$ as follows. 
If $0\le y\le\sigma_s(r-1)$, set $\Phi_{s,r}(y)=0$. 
If $\sigma_s(r-1)<y<1/s!$, choose $t\ge r$ with $y\in[\sigma_s(t-1),\sigma_s(t)]$, let $a\in[0,1/(t-1)]$ be the unique parameter satisfying $y=P_s(t,a)$, and set $\Phi_{s,r}(y):=P_r(t,a)$. 
At common endpoints adjacent branches agree. 
Finally set $\Phi_{s,r}(1/s!)=1/r!$. Equivalently, on the positive range $\Phi_{s,r}=F_r\circ F_s^{-1}$, where $F_m:=\Phi_{2,m}$ and $F_s^{-1}$ is taken on the strictly increasing positive range of $F_s$.
\end{definition}

We shall use Reiher's theorem in the following graphon form.

\begin{theorem}[Reiher~\cite{Reiher2016}]\label{thm:reiher}
For every $m\ge3$ and graphon $W$, one has $g_m(W)\ge F_m(g_2(W))$. 
\end{theorem}

Our main result is the corresponding clique-constrained theorem.

\begin{theorem}\label{thm:main}
Let $2\le s<r$. 
For every graphon $W$, one has $g_r(W)\ge \Phi_{s,r}(g_s(W))$.
\end{theorem}

The bound is sharp throughout the domain. 
On a positive branch, equality is attained by the complete $t$-partite graphon with $t-1$ parts of measure $(1+a)/t$ and one part of measure $(1-(t-1)a)/t$. 
The endpoint $g_s=1/s!$ is attained by the complete graphon $W\equiv1$. 
On the zero interval, every value $0\le y\le\sigma_s(r-1)$ is attained with no $K_r$ by placing a balanced complete $(r-1)$-partite graphon on a measurable set of measure $(y/\sigma_s(r-1))^{1/s}$ and putting all remaining edges equal to zero.

We also prove stability statements. 
For Reiher's original edge-density problem, stability was obtained by Pikhurko and Razborov~\cite{PikhurkoRazborov2017} for triangles, and by Kim, Liu, Pikhurko and Sharifzadeh~\cite{KimLiuPikhurkoSharifzadeh2020} for larger cliques. 
For $m\ge3$ and $\gamma\in[0,1/2]$, let
\[
R_m(\gamma):=\{U:g_2(U)=\gamma\text{ and }g_m(U)=F_m(\gamma)\}
\]
be the graphon equality family in Reiher's clique density theorem.
Concretely, in the zero range $\gamma\le (m-2)/(2(m-1))$, this family consists of the $K_m$-free graphons with edge density $\gamma$. 
In the positive range, choose $k\ge m-1$ and $\beta\in[0,1]$ such that
        $\gamma=\frac{\binom{k}{2}+k\beta}{(k+\beta)^2}.$
Then the members of $R_m(\gamma)$ are, up to measure-preserving relabelling, obtained as follows. 
Partition the space as
        $A_1\cup\cdots\cup A_{k-1}\cup X,$
where $\mu(A_i)=1/(k+\beta)$ for every $i$ and $\mu(X)=(1+\beta)/(k+\beta)$. 
Put all edges between distinct parts of this partition, put no edges inside each $A_i$, and put inside $X$ an arbitrary triangle-free graphon whose edge density, after normalising $X$ to have measure one, is $\beta/(1+\beta)^2$.

For finite graphs, let $\mathcal F_{m,n}$ be the class of all $K_m$-free graphs on $n$ vertices. 
Define $H^+_{m,n}$ as follows. 
Choose an integer $k\ge m-1$ and an integer $q$ with $n/(k+1)<q\le n/k$, and put $b:=n-kq$. 
Partition the vertex set as
        $A_1\cup\cdots\cup A_{k-1}\cup B\cup C,$
where $|A_i|=|B|=q$ for every $i$ and $|C|=b$. 
Start with the complete $k$-partite graph with parts $A_1,\ldots,A_{k-1}, B\cup C,$
and then insert inside the exceptional part $B\cup C$ an arbitrary triangle-free graph with exactly $qb$ edges. 
Let $H^+_{m,n}$ be the family of all graphs obtained in this way that contain a copy of $K_m$. 
Finally set $H_{m,n}:=H^+_{m,n}\cup\mathcal F_{m,n}$.

\begin{theorem}[Pikhurko--Razborov~\cite{PikhurkoRazborov2017}; Kim--Liu--Pikhurko--Sharifzadeh~\cite{KimLiuPikhurkoSharifzadeh2020}]\label{thm:KLP}
Let $m\ge3$. 
For every $\eps>0$ there exist $\delta>0$ and $n_0$ such that every graph $G$ on $n\ge n_0$ vertices satisfying $g_m(G)\le F_m(g_2(G))+\delta$ can be made isomorphic to a member of $H_{m,n}$ by changing at most $\eps n^2$ adjacencies.
\end{theorem}

We next formulate stability for the same $K_s\to K_r$ profile in two settings: a local graphon version on each smooth positive branch, and a finite edit-distance version covering all branches. For a family $\mathcal A$ of graphons, write $\delta_\cut(W,\mathcal A):=\inf_{U\in\mathcal A}\delta_\cut(W,U)$.

\begin{theorem}[Graphon stability]\label{thm:graphon-stability}
Let $3\le s<r$ and $t\ge r$, and let $y=P_s(t,a)$ for some $0<a<1/(t-1)$. 
Set $\gamma_y:=P_2(t,a)$. 
For every $\eps>0$ there exists $\delta>0$ such that every graphon $W$ satisfying $|g_s(W)-y|\le\delta$ and $g_r(W)\le\Phi_{s,r}(g_s(W))+\delta$ also satisfies $\delta_\cut(W,R_s(\gamma_y))<\eps$.
\end{theorem}

For finite graphs, in the positive range the target family is $H^+_{s,n}$, while in the zero range it is the class $\mathcal F_{r,n}$ of all $K_r$-free graphs on $n$ vertices. 
Put $H_{s,r,n}:=H^+_{s,n}\cup\mathcal F_{r,n}$.

\begin{theorem}[Finite edit-distance stability]\label{thm:finite-stability}
Let $3\le s<r$. 
For every $\eps>0$ there exist $\delta>0$ and $n_0$ such that every graph $G$ on $n\ge n_0$ vertices satisfying $g_r(G)\le\Phi_{s,r}(g_s(G))+\delta$ can be made isomorphic to a member of $H_{s,r,n}$ by changing at most $\eps n^2$ adjacencies.
\end{theorem}

We conclude the introduction by outlining the proofs.

\medskip
\noindent\textbf{Proof sketch.} 
Our argument is inspired by the variational-and-link philosophy used by
Kim, Liu, Pikhurko and Sharifzadeh~\cite{KimLiuPikhurkoSharifzadeh2020}
in their structural stability proof for Reiher's clique density theorem. The proof of Theorem~\ref{thm:main} proceeds by induction on $s$, with Reiher's theorem as the base case. 
The zero interval is accounted for by $K_r$-free multipartite examples, and the endpoint $g_s=1/s!$ forces $W=1$ almost everywhere. 
Thus the main issue is the positive range. There is a first technical obstacle at the nonsmooth critical values $g_s=\sigma_s(t)$, where two adjacent multipartite branches meet and the profile is not differentiable. 
We handle these points before the variational argument, using a recursion between adjacent clique densities. 
This recursion shows that once the $K_s$-density reaches the balanced $t$-partite value, all higher clique densities must be at least their balanced $t$-partite values.

It remains to treat a smooth positive branch. 
Assume, for contradiction, that a graphon $W$ minimises $g_r-\Phi_{s,r}(g_s)$ with negative value. 
The main difficulty is that the constraint fixes a nonlinear clique density rather than the edge density, so one cannot directly apply Reiher's theorem to obtain the desired lower bound. 
Instead, first variations force additional structure. A variation of the underlying measure gives a vertex equation relating $W_x(K_{r-1})$ and $W_x(K_{s-1})$, while a one-sided edge variation gives a complementary inequality on pairs. 
These Lagrange conditions are the mechanism by which near-extremality under the $K_s$-constraint begins to recover the missing edge-density information.

The next step is to pass to vertex links. 
After normalising by the degree of $x$ and by the branch parameters, the link of a typical vertex becomes an instance of the lower-order problem $K_{s-1}\to K_{r-1}$. 
The induction hypothesis gives a local profile inequality inside each link. 
This inequality still contains local normalising parameters, and the analytic core of the proof is a one-variable lemma which converts it into a lower bound on the degree of $x$. 
Combinatorially, this says that a hypothetical counterexample cannot hide its $K_s$'s in links of too small degree; the multipartite link is already the most efficient local configuration. Averaging the local degree bounds gives a global lower bound on the average degree of $W$. 
Reiher's original $K_2\to K_s$ theorem then enters in the opposite direction: since $g_s(W)$ is fixed at the value of the chosen multipartite branch, the edge density of $W$ cannot exceed the corresponding multipartite edge density. 
The lower bound from the link analysis and the upper bound from Reiher are incompatible, by a final scalar inequality in the branch parameter. 
This contradiction proves the main theorem.

For stability, we connect near extremality for the new higher-order constraint to the known stability theory for the edge-constrained problem. 
The key step is an equality reduction: every exact positive-range equality case for the $K_s\to K_r$ profile is already an equality case for Reiher's $K_2\to K_s$ theorem at the corresponding edge density. 
This is proved by running the variational/link argument at equality; on smooth branches it forces equality in the edge-density bound, while critical points are handled by the adjacent-clique recursion and induction. 
Once this reduction is available, stability follows by compactness: any sequence of near extremisers has an exact equality limit, and the equality reduction identifies that limit with the desired equality family.

\medskip
\noindent\textbf{Organisation.} 
Section~\ref{sec:critical} collects some basic graphon computations and then treats the balanced multipartite critical values. 
Section~\ref{sec:properties} proves the analytic estimates for the profile. 
The proof of Theorem~\ref{thm:main} is given in Section~\ref{sec:proof-main}. 
Section~\ref{sec:stability} proves Theorems~\ref{thm:graphon-stability} and~\ref{thm:finite-stability}.

\section{Critical values}\label{sec:critical}

We first collect some elementary facts about the profile $\Phi_{s,r}$ and the link densities. We identify graphons that differ by a measure-preserving transformation; the quantities $g_j$ are unchanged under this identification. The endpoint identities
\(P_j(u,0)=\sigma_j(u)\) and  \(P_j\left(u,\frac1{u-1}\right)=\sigma_j(u-1)
\)
show that adjacent branches in the definition of $\Phi_{s,r}$ fit together continuously. Moreover,
\[
 \frac{\partial}{\partial a}P_j(u,a)
 =-j(j-1)\bino{u}{j}u^{-j}a(1+a)^{j-2}.
\]
Thus, for fixed $u$ and $j\ge2$, $P_j(u,a)$ is strictly decreasing in $a$ on $[0,\frac{1}{u-1}]$. Hence $\Phi_{s,r}$ is well-defined and continuous. On a smooth positive branch, say $y=P_s(t,a)$ with $0<a<\frac{1}{t-1}$, the chain rule gives
\begin{equation}\label{eq:branch-derivative}
 \Phi'_{s,r}(y)=\frac{\frac{d}{da}P_r(t,a)}{\frac{d}{da}P_s(t,a)}
 =\frac{r(r-1)}{s(s-1)}\frac{\sigma_r(t)}{\sigma_s(t)}(1+a)^{r-s}.
\end{equation}
The endpoint values $\sigma_s(t)$ and $\sigma_r(t)$ increase with $t$; hence the same monotonicity holds across consecutive branches. In particular, $\Phi_{s,r}$ is strictly increasing on its positive range.

Empty products and integrals over $\Omega^0$ are taken to be $1$. For $j\ge2$, define
\[
 W_x(K_{j-1})=\frac1{(j-1)!}\int_{\Omega^{j-1}}\prod_{1\le a<b\le j-1}W(y_a,y_b)
 \prod_{a=1}^{j-1}W(x,y_a)\,d\mu^{j-1},
\]
and
\[
 W_{x,y}(K_{j-2})=\frac1{(j-2)!}\int_{\Omega^{j-2}}\prod_{1\le a<b\le j-2}W(z_a,z_b)
 \prod_{a=1}^{j-2}W(x,z_a)W(y,z_a)\,d\mu^{j-2}.
\]
By Fubini,
\begin{equation}\label{eq:vertex-fubini}
 \int_\Omega W_x(K_{j-1})\,d\mu(x)=j g_j(W),
\end{equation}
and, for almost every $x$,
\begin{equation}\label{eq:edge-fubini}
 \int_\Omega W(x,y)W_{x,y}(K_{j-2})\,d\mu(y)=(j-1)W_x(K_{j-1}).
\end{equation}
We use Theorem~\ref{thm:reiher} in graphon form throughout the proof. This is equivalent to Reiher's finite asymptotic theorem by the standard compactness and approximation theory of graphons; see, for example,~\cite{Lovasz2012}. For $m\ge3$, the branch parametrisation in Theorem~\ref{thm:reiher} shows that $F_m=\Phi_{2,m}$ is continuous and strictly increasing on its positive range.

\medskip

We now proceed to treat the critical values case. The candidate profile is not differentiable at the balanced multipartite values. We deal with these values before entering the variational argument. This step is the analogue of the treatment of critical values in the classical clique density theorem.

\begin{lemma}\label{lem:adjacent-recursion}
For every graphon $W$ and every integer $m\ge2$, writing $k_j=g_j(W)$,
\[
 m^2 k_m^2\le k_{m-1}\bigl(k_m+(m^2-1)k_{m+1}\bigr).
\]
\end{lemma}

\begin{proof}
For $z=(z_1,\dots,z_{m-1})$, set
 $A(z)=\prod_{1\le a<b\le m-1}W(z_a,z_b)$ and  $\eta(z)=\int_\Omega\prod_{a=1}^{m-1}W(x,z_a)\,d\mu(x)$.
Then
\[
 k_{m-1}=\frac1{(m-1)!}\int A,
 \qquad \text{and } \quad
 m k_m=\frac1{(m-1)!}\int A\eta.
\]
By Cauchy--Schwarz, applied to $A^{1/2}$ and $A^{1/2}\eta$ on $\Omega^{m-1}$, we get
\begin{equation}\label{eq:cauchy-adjacent}
 m^2 k_m^2\le k_{m-1}\cdot \frac1{(m-1)!}\int A\eta^2.
\end{equation}
It remains to bound the last factor. Since the space is atomless, diagonal contributions have measure zero. After symmetrising over an ordered $(m+1)$-tuple $u=(u_1,\dots,u_{m+1})$, we get
\[
 \frac1{(m-1)!}\int A\eta^2=\frac2{(m+1)!}\int_{\Omega^{m+1}}B(u)\,d\mu^{m+1}(u),
\]
where, for $a_{ij}=W(u_i,u_j)$,
$B(u)=\sum_{e\in\binom{[m+1]}2}\prod_{f\in\binom{[m+1]}2\setminus\{e\}}a_f.$
Here the factor $\frac{2}{(m+1)!}$ comes from the symmetrisation. After the original $(m-1)$-tuple and the two vertices from $\eta^2$ are arranged as an ordered $(m+1)$-tuple, the missing pair is an unordered edge, while the two vertices from $\eta^2$ may be ordered in two ways. Also
\[
 k_m=\frac1{(m+1)!}\int S(u), \qquad k_{m+1}=\frac1{(m+1)!}\int K(u),
\]
where
 $S(u)=\sum_{T\in\binom{[m+1]}m}\prod_{e\in\binom{T}2}a_e$ and 
 $K(u)=\prod_{e\in\binom{[m+1]}2}a_e$.
The pointwise inequality
\(
 2B(u)\le S(u)+(m^2-1)K(u)
\)
holds for all $a_e\in[0,1]$. Both sides are multilinear in the edge variables, so it suffices to check $a_e\in\{0,1\}$. If at least two edges are missing, then $K=B=0$. If exactly one edge is missing, then $K=0$, $B=1$, and $S=2$. If no edge is missing, then $K=1$, $B=\binom{m+1}{2}$, and $S=m+1$, so $2B-S=m^2-1$. Integrating the pointwise inequality gives
\[
 \frac1{(m-1)!}\int A\eta^2\le k_m+(m^2-1)k_{m+1}.
\]
Combining this with~\eqref{eq:cauchy-adjacent} proves the lemma.
\end{proof}

\begin{lemma}\label{lem:critical-values}
Let $3\le s<r\le t$. Assume Theorem~\ref{thm:main} is known for source clique size $s-1$ and all larger target clique sizes. For every graphon $W$, if $g_s(W)=\sigma_s(t)$, then $g_r(W)\ge\sigma_r(t)$.
\end{lemma}

\begin{proof}
Write $k_j=g_j(W)$. Since $\Phi_{s-1,s}(\sigma_{s-1}(t))=\sigma_s(t)$, the induction hypothesis and the strict monotonicity of $\Phi_{s-1,s}$ on its positive range imply that $k_{s-1}>\sigma_{s-1}(t)$ would force $k_s>\sigma_s(t)$, a contradiction. Hence $k_{s-1}\le\sigma_{s-1}(t)$. Since $k_s=\sigma_s(t)>0$, we must have $k_{s-1}>0$. Put
\(
 z_m:=\frac{m k_m}{k_{m-1}}.
\)
Then $z_s$ is defined and
\(
 z_s\ge\frac{s\sigma_s(t)}{\sigma_{s-1}(t)}=\frac{t-s+1}{t}.
\)
Lemma~\ref{lem:adjacent-recursion} gives, after division by $k_{m-1}k_m$ whenever these quantities are positive,
\(
 m z_m\le 1+(m-1)z_{m+1},
\)
equivalently $z_{m+1}\ge \frac{m z_m-1}{m-1}$. Starting from the displayed positive lower bound for $z_s$, this recursion gives successively
\(
 z_m\ge\frac{t-m+1}{t}>0\) for \(s\le m\le r.
\)
In particular $k_m>0$ for $s\le m\le r$, and all quotients used in the recursion are legitimate. Consequently
\[
 k_r=k_s\prod_{m=s}^{r-1}\frac{k_{m+1}}{k_m}
 \ge \binom{t}{s}t^{-s}\prod_{m=s}^{r-1}\frac{t-m}{t(m+1)}
 =\binom{t}{r}t^{-r}.
\]
This proves the lemma.
\end{proof}

\section{Properties of the multipartite profiles}\label{sec:properties}

We record here the analytic estimates for the multipartite profile functions $\Phi_{m,q}$ that will be used in the proof. Throughout this section, $N$ denotes an auxiliary profile parameter, while $n$ is reserved for the number of vertices in a finite graph. For integers $N\ge j-1$ and real $b$ with $N+b>0$, set
\[
 D_j(N,b):=\frac{\binom{N}{j}+\binom{N}{j-1}b}{(N+b)^j}.
\]
For $0\le b\le1$, this is the $K_j$-density of the complete $(N+1)$-partite graphon with $N$ equal parts of measure $(N+b)^{-1}$ and one part of measure $b(N+b)^{-1}$. The two parametrisations are related by
\[
 D_j(N,\beta)=P_j\left(N+1,\frac{1-\beta}{N+\beta}\right),\qquad 0\le\beta\le1.
\]
Thus the positive branches of $\Phi_{m,q}$ may equivalently be parametrised by $D_m(N,\beta)\mapsto D_q(N,\beta)$.

\begin{lemma}\label{lem:Q-monotone}
Let $2\le m<q$. Define
\(
 Q_{m,q}(y):=\frac{\Phi_{m,q}(y)}{y^{\frac{q}{m}}}
\)
on the positive range of $\Phi_{m,q}$, and set $Q_{m,q}(y)=0$ on the zero interval. Then $Q_{m,q}$ is nondecreasing on its domain and strictly increasing on the positive range of $\Phi_{m,q}$.
\end{lemma}

\begin{proof}
On a positive branch write $y=D_m(N,\beta)$, and $\Phi_{m,q}(y)=D_q(N,\beta)$ for $0\le\beta\le1$. For $j\in\{m,q\}$, put $\alpha_j=\frac{N-j+1}{j}$. Away from the endpoint at which $D_q(N,\beta)=0$, we may differentiate logarithms. Then
 $D_j(N,\beta)=\binom{N}{j-1}\frac{\alpha_j+\beta}{(N+\beta)^j}$, and hence  $\frac{d}{d\beta}\log D_j(N,\beta)=\frac{(j-1)(1-\beta)}{(N+\beta)(\alpha_j+\beta)}$.
Therefore
\[
 \frac{d}{d\beta}\log\left(\frac{D_q(N,\beta)}{D_m(N,\beta)^{\frac{q}{m}}}\right)
 =\frac{(q-m)(1-\beta)}{m(\alpha_q+\beta)(\alpha_m+\beta)}\ge0.
\]
Since $D_m(N,\beta)$ is increasing in $\beta$, the quotient $Q_{m,q}$ is nondecreasing on each branch. On the interior of a positive branch the displayed derivative is strictly positive. The endpoint cases follow by continuity. Adjacent endpoint values agree, and hence $Q_{m,q}$ is strictly increasing on the positive range and constant only on the zero interval.
\end{proof}

\begin{lemma}\label{lem:support}
Let $2\le m<q$. If $N\ge q-1$ and
\(
 -\frac{N-q+1}{q}\le b\le0,
\)
then $D_m(N,b)$ lies in the domain of $\Phi_{m,q}$ and
\(
 \Phi_{m,q}(D_m(N,b))\ge D_q(N,b).
\)
\end{lemma}

\begin{proof}
Put $y=N+b$ and $f_j^{(N)}(y):=\binom{N}{j-1}\frac{y-\frac{(j-1)(N+1)}{j}}{y^j}.$ Then $D_j(N,b)=f_j^{(N)}(y)$, and the hypothesis on $b$ is equivalent to
\begin{equation}\label{eq:support-range}
 \frac{(q-1)(N+1)}q\le y\le N.
\end{equation}
If $N=q-1$, then the interval in the hypothesis is the singleton $b=0$. Thus $D_q(q-1,0)=0$ and $D_m(q-1,0)=\sigma_m(q-1)$, which is the endpoint of the zero interval of $\Phi_{m,q}$; the conclusion follows directly. We may therefore assume $N\ge q$.

We claim that for every $y$ satisfying~\eqref{eq:support-range} there is a unique
\(
 Y\in\left[\frac{(q-1)N}{q},N\right]
\)
such that
\begin{equation}\label{eq:Ym-equal}
 f_m^{(N-1)}(Y)=f_m^{(N)}(y) \quad \text{and}\quad f_q^{(N-1)}(Y)\ge f_q^{(N)}(y).
\end{equation}
Indeed,
$ \frac{d}{dy}f_j^{(N)}(y)=\binom{N}{j-1}\frac{(j-1)(N+1-y)}{y^{j+1}},$
so $f_j^{(N)}$ is increasing on $(0,N]$. With $Y_0=\frac{Ny}{N+1}$, one has $Y_0\ge\frac{(q-1)N}{q}$ and
\[
 \frac{f_m^{(N-1)}(Y_0)}{f_m^{(N)}(y)}
 =\frac{N-m+1}{N}\left(\frac{N+1}{N}\right)^{m-1}
 \le \left(1-\frac1{N+1}\right)^{m-1}\left(\frac{N+1}{N}\right)^{m-1}
 =1.
\]
Moreover $f_m^{(N-1)}(N)=f_m^{(N)}(N)\ge f_m^{(N)}(y)$. Thus the first part of~\eqref{eq:Ym-equal} has a unique solution $Y$ in the required interval.

Regard $Y$ as a function of $y$ and set
\(
 G(y):=f_q^{(N-1)}(Y(y))-f_q^{(N)}(y).
\)
Then $G(N)=0$. Implicit differentiation shows that $G'(y)\le0$ is equivalent to
\(
 \left(\frac yY\right)^{q-m}\le\frac{N-m+1}{N-q+1}.
\)
Since $Y\ge \frac{Ny}{N+1}$, with $d=q-m<N$ we have
\[
 \left(\frac yY\right)^d\le\left(1+\frac1N\right)^d\le\frac{N}{N-d}\le\frac{N-m+1}{N-q+1},
\]
the last inequality being equivalent to $(q-m)(m-1)\ge0$. Hence $G'(y)\le0$, and so $G(y)\ge G(N)=0$, proving the second part of~\eqref{eq:Ym-equal}.

Set $b_1=Y-(N-1)$. Since $N\ge Y\ge\frac{(q-1)N}{q}$, we have
 $1\ge b_1\ge\frac{(q-1)N}{q}-(N-1)=-\frac{N-q}{q},$
so if $b_1<0$ the same hypothesis is valid with $N$ replaced by $N-1$. Moreover
\[
 D_m(N-1,b_1)=D_m(N,b),\qquad D_q(N-1,b_1)\ge D_q(N,b).
\]
If $b_1\in[0,1]$, the point lies on a genuine multipartite branch and the conclusion follows from the definition of $\Phi_{m,q}$. If $b_1<0$, repeat the compression step. At each step the first parameter decreases by one, $D_m$ is preserved, and $D_q$ does not decrease. The process must terminate after finitely many steps, because at $N=q-1$ the admissible negative interval has collapsed to the single point $b=0$. Hence for some $N'\ge q-1$ and some $\beta\in[0,1]$,
 $D_m(N',\beta)=D_m(N,b)$, $D_q(N',\beta)\ge D_q(N,b),$
and therefore $\Phi_{m,q}(D_m(N,b))\ge D_q(N,b)$.
\end{proof}

Let $2\le m<q$, put
\(
 L:=\frac{q}{q-m},
\)
and let $t\ge q+1$. Define
\(
 c:=\frac1L\binom{t-1}{m}\), and \(h:=\frac{q-m}{m}\binom{t-1}{q}.
\)

\begin{lemma}\label{lem:one-variable}
Assume $M>1$ and $M\le \eta\le LM$. Let $\omega>0$, and suppose that $\frac{c\eta}{\omega^m}$ lies in the domain of $\Phi_{m,q}$ and that
\begin{equation}\label{eq:one-var-hyp}
 \frac{h(\eta-M)}{\omega^q}\ge \Phi_{m,q}\left(\frac{c\eta}{\omega^m}\right).
\end{equation}
Then
$ \omega\ge\frac{t-1}{m}M^{\frac{1}{m}-1}\left(\frac\eta L+(m-1)M\right).$
\end{lemma}

\begin{proof}
Set
\(
 u=\frac\eta M\), and 
 \( y=\frac\omega{M^{\frac{1}{m}}}.
\)
Then $1\le u\le L$. It suffices to show
\[
 y\ge y_0:=\frac{t-1}{m}\left(\frac uL+m-1\right).
\]
Assume $y<y_0$. Let
\(
 \delta=\frac{cu}{y^m}\), and \(\delta_0=\frac{cu}{y_0^m}.
\)
Put $N=t-2$ and $b=y_0-N$. Since $L=\frac{q}{q-m}$, we may write
\[
 y_0=\frac{t-1}{mq}\bigl((q-m)u+q(m-1)\bigr), \quad\text{and hence}\quad b=\frac{(t-1)(q-m)u-q(t-m-1)}{mq}.
\]
Because $1\le u\le L$, this gives
\(
 -\frac{N-q+1}{q}\le b\le1.
\)
To compute $D_m(N,b)$ and $D_q(N,b)$, we use
 $D_j(N,b)=\binom{N}{j-1}\frac{b+\frac{N-j+1}{j}}{y_0^j}$ and $N+b=y_0$.
For $N=t-2$ we have
 $b+\frac{N-m+1}{m}=\frac{t-1}{mL}u$ and
 $b+\frac{N-q+1}{q}=\frac{(t-1)(q-m)}{mq}(u-1).$
Together with
 $\binom{t-2}{m-1}\frac{t-1}{m}=\binom{t-1}{m}
$ and 
 $\binom{t-2}{q-1}\frac{t-1}{q}=\binom{t-1}{q},$
this gives
\begin{equation}\label{eq:delta0-D}
 \delta_0=D_m(t-2,b),\qquad D_q(t-2,b)=\frac{h(u-1)}{y_0^q}.
\end{equation}
The displayed range for $b$ also shows that $\delta_0$ lies in the domain of $\Phi_{m,q}$: if $b\le0$, this follows from Lemma~\ref{lem:support}, while if $0\le b\le1$, it follows from the genuine branch parametrisation. Since $\delta>\delta_0$, Lemma~\ref{lem:Q-monotone} gives
\[
 y^q\Phi_{m,q}(\delta)=(cu)^{\frac{q}{m}}Q_{m,q}(\delta)
 \ge (cu)^{\frac{q}{m}}Q_{m,q}(\delta_0)=y_0^q\Phi_{m,q}(\delta_0).
\]
If $b\le0$, use Lemma~\ref{lem:support}; if $0\le b\le1$, use the genuine branch identity. In both cases~\eqref{eq:delta0-D} yields
 $y^q\Phi_{m,q}(\delta)\ge h(u-1).$
On the other hand,~\eqref{eq:one-var-hyp} becomes
$y^q\Phi_{m,q}(\delta)\le h(u-1)M^{1-\frac{q}{m}}$ and thus $M^{1-\frac{q}{m}}\ge 1$.
If $u>1$, then $h(u-1)>0$ while $M^{1-\frac{q}{m}}<1$, a contradiction.

It remains to consider $u=1$. Then~\eqref{eq:one-var-hyp} gives
\(
0\ge\Phi_{m,q}\left(\frac{c}{y^m}\right).
\)
Since $\Phi_{m,q}$ is nonnegative and vanishes exactly on its zero interval, this implies
\(
 \frac{c}{y^m}\le \sigma_m(q-1)=\frac{\binom{q-1}{m}}{(q-1)^m}.
\)
Thus
\(
 y^m\ge \frac{c(q-1)^m}{\binom{q-1}{m}}.
\)
For $u=1$,
$ y_0=\frac{t-1}{m}\left(\frac1L+m-1\right)=\frac{(t-1)(q-1)}{q}.$
It is therefore enough to verify
\[
 \frac{c(q-1)^m}{\binom{q-1}{m}}
 \ge \left(\frac{(t-1)(q-1)}{q}\right)^m.
\]
After cancelling $(q-1)^m$ and using $c=(q-m)q^{-1}\binom{t-1}{m}$, this is equivalent to
\(
 \frac{\binom{t-1}{m}}{(t-1)^m}\ge\frac{\binom{q}{m}}{q^m}.
\)
This follows because $N\mapsto\frac{\binom{N}{m}}{N^m}$ is increasing for $N\ge m$ and $t-1\ge q$.
\end{proof}

\begin{lemma}\label{lem:one-variable-M-one}
The conclusion of Lemma~\ref{lem:one-variable} remains true when the hypothesis $M>1$ is replaced by $M=1$.
\end{lemma}

\begin{proof}
Repeat the proof of Lemma~\ref{lem:one-variable} with $M=1$. Then $u=\eta\in[1,L]$ and $y=\omega$. If the asserted lower bound fails, then $y<y_0$, and the same notation gives $\delta=\frac{cu}{y^m}>\delta_0=\frac{cu}{y_0^m}$.

If $u>1$, then $D_q(t-2,b)=\frac{h(u-1)}{y_0^q}>0$. Moreover $\delta_0$ lies in the positive range of $\Phi_{m,q}$: for $b\le0$ this follows from Lemma~\ref{lem:support}, since otherwise $\Phi_{m,q}(\delta_0)=0$; for $0\le b\le1$ it follows from the genuine branch identity. By the strict part of Lemma~\ref{lem:Q-monotone},
\(
 y^q\Phi_{m,q}(\delta)>y_0^q\Phi_{m,q}(\delta_0).
\)
By Lemma~\ref{lem:support} when $b\le0$, and by the genuine branch identity when $0\le b\le1$, we have $y_0^q\Phi_{m,q}(\delta_0)\ge h(u-1)$. Hence $y^q\Phi_{m,q}(\delta)>h(u-1)$, whereas the assumed profile inequality with $M=1$ gives $y^q\Phi_{m,q}(\delta)\le h(u-1)$, a contradiction. If $u=1$, the zero-range argument in the proof of Lemma~\ref{lem:one-variable} applies without change. Therefore $y\ge y_0$.
\end{proof}

\section{Clique profiles}\label{sec:proof-main}

We are ready to prove Theorem~\ref{thm:main}. The argument has five steps: extremality, normalisation of the local parameters, induction inside vertex links, the global edge-density bound, and the final scalar comparison.

\begin{proof}[Proof of Theorem~\ref{thm:main}]
We proceed by induction on $s$. The case $s=2$ is Theorem~\ref{thm:reiher}. Assume $s\ge3$ and that the theorem is known for source clique size $s-1$ and all larger target sizes. Fix $r>s$. Suppose, for contradiction, that a counterexample exists. Since graphon space modulo measure-preserving transformations is compact in the cut metric and clique densities are cut-continuous, the continuous functional
\[
 H(W):=g_r(W)-\Phi_{s,r}(g_s(W))
\]
has a minimiser with $H(W)<0$. Write $k_j=g_j(W)$.

The zero range $k_s\le\sigma_s(r-1)$ is impossible because then $\Phi_{s,r}(k_s)=0$. Also $k_s\ne\frac{1}{s!}$. Indeed, if $g_s(W)=\frac{1}{s!}$, then the integral defining $s!g_s(W)$ is equal to $1$; since this integrand is bounded above by $1$, it equals $1$ for almost every $s$-tuple, and Fubini then gives $W=1$ almost everywhere. Hence in that case $g_r(W)=\frac{1}{r!}$, not a counterexample. By Lemma~\ref{lem:critical-values}, $k_s$ is not a critical value. Therefore for some $t\ge r$ and some $0<a<\frac{1}{t-1}$,
\(
 k_s=P_s(t,a)\),
 \(k_r<P_r(t,a).
\)
Put
\(
 \rho:=\frac{1+a}{t},
\)
and let $\lambda=\frac{r(r-1)}{s(s-1)}\frac{\sigma_r(t)}{\sigma_s(t)}(1+a)^{r-s}=\Phi'_{s,r}(k_s)$ as in~\eqref{eq:branch-derivative}. Set $p_r:=P_r(t,a)$ and
\(
 \tau_0:=s\lambda k_s-rp_r.
\)
We compute $\tau_0$ on this branch explicitly. Recall
$k_s=\sigma_s(t)(1+a)^{s-1}\bigl(1-(s-1)a\bigr)$ and $p_r=\sigma_r(t)(1+a)^{r-1}\bigl(1-(r-1)a\bigr).$
Therefore
\begin{align*}
 s\lambda k_s-rp_r
 &=\frac{r(r-1)}{s-1}\sigma_r(t)(1+a)^{r-1}\bigl(1-(s-1)a\bigr) -r\sigma_r(t)(1+a)^{r-1}\bigl(1-(r-1)a\bigr) \\
 &=\frac{r(r-s)}{s-1}\sigma_r(t)(1+a)^{r-1}.
\end{align*}
Thus
\begin{equation}\label{eq:tau0}
 \tau_0=\frac{r(r-s)}{s-1}\binom{t}{r}t^{-r}(1+a)^{r-1}>0.
\end{equation}

\medskip
\noindent\emph{Step 1: extremality.}
Define
\(
 C(x):=W_x(K_{s-1})\), and \( D(x):=W_x(K_{r-1}).
\)
We first claim that there is a real number $\tau$ such that
\begin{equation}\label{eq:vertex-equation}
 D(x)=\lambda C(x)-\tau
\end{equation}
for almost every $x$. To see this, take $\phi\in L^\infty(\Omega)$ with $\int\phi\,d\mu=0$ and set
\(
 d\mu_\eps=(1+\eps\phi)d\mu
\)
for sufficiently small $|\eps|$. Taking $|\eps|<\|\phi\|_\infty^{-1}$ when $\phi$ is nonzero, $\mu_\eps$ is a probability measure equivalent to $\mu$. In particular, $(\Omega,\mu_\eps)$ is again a standard atomless probability space. After passing, if desired, to an isomorphic standard probability model, the triple $(\Omega,\mu_\eps,W)$ is therefore an admissible graphon competitor. Differentiating gives
\[
 \left.\frac{d}{d\eps}g_j(W;\mu_\eps)\right|_{\eps=0}
 =\int_\Omega \phi(x)W_x(K_{j-1})\,d\mu(x).
\]
Minimality of $W$ and differentiability of $\Phi_{s,r}$ at $k_s$ imply
$ \int_\Omega \phi(x)(D(x)-\lambda C(x))\,d\mu(x)=0$
for every mean-zero $\phi$. Hence $D-\lambda C$ is almost everywhere constant, which proves~\eqref{eq:vertex-equation}.

Next, let $B\in L^\infty(\Omega^2)$ be symmetric with $0\le B\le W$. The downward perturbation $W_\eps=W-\eps B$ is admissible for small $\eps>0$, and
\[
 \left.\frac{d}{d\eps}g_j(W+\eps B)\right|_{\eps=0}
 =\frac12\int_{\Omega^2} B(x,y)W_{x,y}(K_{j-2})\,d\mu(x)d\mu(y).
\]
Applying this derivative to $W-\eps B$ and using minimality gives
\[
 \int_{\Omega^2} B(x,y)\bigl(W_{x,y}(K_{r-2})-\lambda W_{x,y}(K_{s-2})\bigr)\,d\mu(x)d\mu(y)\le0.
\]
Taking $B=W\mathbf 1_E$ and varying over measurable symmetric sets $E\subseteq\Omega^2$ gives the pointwise one-sided inequality
$W(x,y)\bigl(W_{x,y}(K_{r-2})-\lambda W_{x,y}(K_{s-2})\bigr)\le0$
for almost every $(x,y)$. Integrating this in $y$ and using~\eqref{eq:edge-fubini} yields
\begin{equation}\label{eq:vertex-ineq}
 (r-1)D(x)\le(s-1)\lambda C(x)
\end{equation}
for almost every $x$. Integrating~\eqref{eq:vertex-equation} and using~\eqref{eq:vertex-fubini}, we obtain $r k_r=s\lambda k_s-\tau.$
Together with $k_r<p_r$, this gives $\tau>\tau_0$. Define
 $M:=\frac{\tau}{\tau_0}>1.$

\medskip
\noindent\emph{Step 2: normalised local parameters.}
Let
\[
 m:=s-1,
 \quad q:=r-1,
 \quad L:=\frac{q}{q-m}=\frac{r-1}{r-s},
\quad\text{and}\quad
 c:=\frac1L\binom{t-1}{m},
 \quad
 h:=\frac{q-m}{m}\binom{t-1}{q}.
\]
The identities relating these constants to the branch parameter are immediate from~\eqref{eq:branch-derivative} and~\eqref{eq:tau0}. First,
\[
 \frac{\tau_0}{\lambda}
 =\frac{s(r-s)}{r-1}\sigma_s(t)(1+a)^{s-1}
 =\frac1L\binom{t-1}{s-1}\left(\frac{1+a}{t}\right)^{s-1}=c\rho^m.
\]
Second,
 $\tau_0=\frac{r-s}{s-1}\binom{t-1}{r-1}\left(\frac{1+a}{t}\right)^{r-1}=h\rho^q.$
Since $D\ge0$, the vertex equation gives $C\ge\frac{\tau}{\lambda}$. Combining~\eqref{eq:vertex-ineq} with $D=\lambda C-\tau$ gives
\(
\frac{\tau}{\lambda}\le C(x)\le L\frac{\tau}{\lambda}
\)
for almost every $x$. Thus there is a measurable function $\eta$ such that
\begin{equation}\label{eq:eta-range}
 M\le\eta(x)\le LM
\end{equation}
and
\begin{equation}\label{eq:CD-eta}
 C(x)=c\rho^m\eta(x),
 \qquad
 D(x)=h\rho^q(\eta(x)-M).
\end{equation}
Averaging the first identity in~\eqref{eq:CD-eta} gives
\begin{equation}\label{eq:eta-average}
 \int_\Omega \eta\,d\mu=L(1-ma).
\end{equation}
Indeed, this is just $\int C=s k_s$ together with the definitions of $c$, $\rho$, and $k_s=P_s(t,a)$.

\medskip
\noindent\emph{Step 3: induction in vertex links.}
Let
\(
 d(x):=\int_\Omega W(x,y)\,d\mu(y)
\)
be the degree function. Since $C(x)>0$ almost everywhere, we have $d(x)>0$ almost everywhere. For such $x$, define the link graphon $W^{(x)}$ on the probability space $(\Omega,\nu_x)$ by
\[
 d\nu_x(y):=\frac{W(x,y)}{d(x)}\,d\mu(y),
 \qquad
 W^{(x)}(y,z):=W(y,z).
\]
Here $\nu_x$ is a probability measure and $\nu_x\ll\mu$; since $\mu$ is atomless, $(\Omega,\nu_x)$ is also atomless. Thus $W^{(x)}$ is an admissible graphon. Put $d(x)=\rho\omega(x)$. By~\eqref{eq:CD-eta}, the clique densities in the link are
\[
 g_m(W^{(x)})=\frac{C(x)}{d(x)^m}=\frac{c\eta(x)}{\omega(x)^m},
 \qquad
 g_q(W^{(x)})=\frac{D(x)}{d(x)^q}=\frac{h(\eta(x)-M)}{\omega(x)^q}.
\]
In particular $\frac{c\eta(x)}{\omega(x)^m}$ belongs to the domain of $\Phi_{m,q}$. Applying the induction hypothesis to $W^{(x)}$ and the pair $(m,q)$ gives
$ \frac{h(\eta(x)-M)}{\omega(x)^q}
 \ge \Phi_{m,q}\left(\frac{c\eta(x)}{\omega(x)^m}\right).$
Together with~\eqref{eq:eta-range}, Lemma~\ref{lem:one-variable} yields
$ \omega(x)\ge\frac{t-1}{m}M^{\frac{1}{m}-1}\left(\frac{\eta(x)}{L}+(m-1)M\right)$
for almost every $x$. Integrating and using~\eqref{eq:eta-average},
\begin{equation}\label{eq:omega-lower}
 \int_\Omega \omega\,d\mu
 \ge\frac{t-1}{m}M^{\frac{1}{m}-1}\bigl(1-ma+(m-1)M\bigr).
\end{equation}

\medskip
\noindent\emph{Step 4: the global edge-density bound.}
Let $\gamma=g_2(W)$. The candidate multipartite graphon with parameters $(t,a)$ has edge density
\(\gamma_0=P_2(t,a)=\frac{t-1}{2t}(1-a^2)
\)
and $K_s$-density $P_s(t,a)=k_s$. If $\gamma>\gamma_0$, then the strict monotonicity of $F_s$ and Theorem~\ref{thm:reiher} give
$ k_s=g_s(W)\ge F_s(\gamma)>F_s(\gamma_0)=k_s,$
a contradiction. Hence $\gamma\le\gamma_0$. Since
$ \int_\Omega d(x)\,d\mu(x)=2g_2(W)=2\gamma$
and $d=\rho\omega$, we get
\begin{equation}\label{eq:omega-upper}
 \int_\Omega \omega\,d\mu\le\frac{2\gamma_0}{\rho}=(t-1)(1-a).
\end{equation}

\medskip
\noindent\emph{Step 5: conclusion.}
Combining~\eqref{eq:omega-lower} and~\eqref{eq:omega-upper} gives
 $\frac{1}{m} M^{\frac{1}{m}-1}\bigl(1-ma+(m-1)M\bigr)\le1-a.$
Equivalently, writing $z=M^{\frac{1}{m}}>1$, we have
\[
 P(z):=(m-1)z^m-m(1-a)z^{m-1}+1-ma\le 0.
\]
But $P(1)=0$ and
 $P'(z)=m(m-1)z^{m-2}\bigl(z-(1-a)\bigr)>0$, a contradiction. The induction step is complete, and hence Theorem~\ref{thm:main} follows.
\end{proof}

\section{Stability}\label{sec:stability}

We now prove Theorems~\ref{thm:graphon-stability} and~\ref{thm:finite-stability}.

\begin{proposition}\label{prop:equality-reduction}
Let $3\le s<r$. Suppose that a graphon $W$ satisfies
\(
 g_s(W)>\sigma_s(r-1)
\) 
and
\(
 g_r(W)=\Phi_{s,r}(g_s(W)).
\)
Then $W\in R_s(\gamma)$, i.e.~if $\gamma=g_2(W)$, then
\(
 g_s(W)=F_s(\gamma).
\)
Moreover, if
\(g_s(W)=P_s(t,a)\) for \(t\ge r\) and  \(0<a<\frac1{t-1}
\),
then
\(
 g_2(W)=P_2(t,a).
\)
\end{proposition}

\begin{proof}
The smooth-branch case is independent of induction. In the critical-value case we argue by induction on $s$; the first nontrivial source size is $s=3$, and for $s>3$ we assume the proposition already known with source size $s-1$.

First suppose that $g_s(W)$ lies in a smooth positive branch, say
\(g_s(W)=P_s(t,a)\) for \(t\ge r\) and  \(0<a<\frac1{t-1}
\). 
Write $k_j=g_j(W)$. By Theorem~\ref{thm:main}, the functional
\(
 U\longmapsto g_r(U)-\Phi_{s,r}(g_s(U))
\)
is nonnegative on graphon space. The assumed equality shows that $W$ is a global minimiser of this functional. Repeating the extremality argument from Step~1 in the proof of Theorem~\ref{thm:main}, with
\(
 C(x)=W_x(K_{s-1})\), \(D(x)=W_x(K_{r-1})\) and \(
 \lambda=\Phi'_{s,r}(g_s(W)),
\)
gives a real number $\tau$ such that
\(
 D(x)=\lambda C(x)-\tau
\)
for almost every $x$, and the one-sided edge variation gives
\[
 W(x,y)\bigl(W_{x,y}(K_{r-2})-\lambda W_{x,y}(K_{s-2})\bigr)\le0
\]
for almost every $(x,y)$. Consequently,
\(
 (r-1)D(x)\le(s-1)\lambda C(x)
\)
for almost every $x$.

We use the notation from the proof of Theorem~\ref{thm:main}. Put
\(
 m=s-1\),
\(q=r-1\),
 \(L=\frac{q}{q-m}\),
 \( \rho=\frac{1+a}{t}\),
\(
 c=\frac1L\binom{t-1}{m},
 \) and
 \(h=\frac{q-m}{m}\binom{t-1}{q}.
\)
Also put
\(
 p_r=P_r(t,a)\), and 
 \(\tau_0=s\lambda P_s(t,a)-rP_r(t,a).
\)
As in the proof of Theorem~\ref{thm:main},
\(
 \frac{\tau_0}{\lambda}=c\rho^m
 \) and
 \(\tau_0=h\rho^q.
\)
Integrating the vertex equation gives
\(
 r k_r=s\lambda k_s-\tau.
\)
Since equality holds and $k_s=P_s(t,a)$, $k_r=P_r(t,a)$, we get $\tau=\tau_0$. Thus the analogue of the parameter $M$ in the proof of Theorem~\ref{thm:main}, namely $M=\frac{\tau}{\tau_0}$, is equal to $1$.

The derivation of the local parameter is unchanged. Hence there is a measurable function $\eta$ such that
\(
 1\le\eta(x)\le L
\),
\( C(x)=c\rho^m\eta(x)
 \) and \(D(x)=h\rho^q(\eta(x)-1).
\)
Moreover,
\(\int_\Omega\eta\,d\mu=L(1-ma).
\)
Let
\(
 d(x):=\int_\Omega W(x,y)\,d\mu(y)
\)
and write $d(x)=\rho\omega(x)$. The link argument from the proof of Theorem~\ref{thm:main}, with Lemma~\ref{lem:one-variable-M-one} replacing Lemma~\ref{lem:one-variable}, gives
\[
 \omega(x)\ge\frac{t-1}{m}\left(\frac{\eta(x)}{L}+m-1\right)
\]
for almost every $x$.
After integration,
\(\int_\Omega\omega\,d\mu\ge(t-1)(1-a).
\)
On the other hand, as in Step~4,
\[
\int_\Omega\omega\,d\mu=\frac{2g_2(W)}{\rho}
 \le\frac{2P_2(t,a)}{\rho}=(t-1)(1-a),
 \qquad \rho=\frac{1+a}{t}.
\]
Together with the preceding lower bound, this forces equality throughout
the displayed inequality above. Hence
\(
 g_2(W)=P_2(t,a),
\)
and therefore
\[
 g_s(W)=P_s(t,a)=F_s(P_2(t,a))=F_s(g_2(W)).
\]
This proves the smooth branch case.

It remains to consider a positive critical value. Thus
\( g_s(W)=\sigma_s(t)
\)
for some $t\ge r$, and equality gives
\(
 g_r(W)=\sigma_r(t).
\)
We first show that equality must also hold at the previous level:
\(
 g_{s-1}(W)=\sigma_{s-1}(t).
\)
Indeed, by Theorem~\ref{thm:main} applied to the pair $(s-1,s)$, together with the strict monotonicity of $\Phi_{s-1,s}$ on its positive range, we have
\(
 g_{s-1}(W)\le\sigma_{s-1}(t).
\)
If this inequality were strict, then $g_{s-1}(W)>0$ and
\(
 z_s=\frac{s g_s(W)}{g_{s-1}(W)}
\)
would be strictly larger than the balanced value $\frac{t-s+1}{t}$. The recursive estimate
\(
 z_{m+1}\ge \frac{m z_m-1}{m-1}
\)
used in Lemma~\ref{lem:critical-values} is strictly increasing in $z_m$, so the strict improvement propagates from $z_s$ to $z_{s+1},\dots,z_r$. Consequently the product formula in the proof of Lemma~\ref{lem:critical-values} gives
\(
 g_r(W)>\sigma_r(t),
\)
contradicting $g_r(W)=\sigma_r(t)$. Hence
\(
 g_{s-1}(W)=\sigma_{s-1}(t).
\)
Therefore
\(
 g_s(W)=\Phi_{s-1,s}(g_{s-1}(W)).
\)
Since $t\ge r>s$, this lies in the positive range for the pair $(s-1,s)$.

If $s=3$, then $g_{s-1}(W)=\sigma_{s-1}(t)$ already says
\(
 g_2(W)=\sigma_2(t).
\)
If $s>3$, the induction hypothesis in the present proposition, applied to the equality case $K_{s-1}\to K_s$, shows that $W$ is an equality case in Theorem~\ref{thm:reiher}. Hence, with $\gamma=g_2(W)$,
\(
 g_{s-1}(W)=F_{s-1}(\gamma).
\)
Since
\(
 g_{s-1}(W)=\sigma_{s-1}(t)=F_{s-1}(\sigma_2(t))
\)
and $F_{s-1}$ is strictly increasing on the positive range, we again get
\(
 g_2(W)=\sigma_2(t).
\)
In either case,
\[ g_s(W)=\sigma_s(t)=F_s(\sigma_2(t))=F_s(g_2(W)).
\]
This proves the critical case.

Finally, if $g_s(W)=\frac{1}{s!}$, then $W=1$ almost everywhere, and the conclusion is immediate with $g_2(W)=\frac{1}{2}$. This completes the proof.
\end{proof}

\begin{proof}[Proof of Theorem~\ref{thm:graphon-stability}]
Suppose the assertion fails. Then there exist $\eps_0>0$, a sequence $\delta_n\to0$, and graphons $W_n$ such that
 $|g_s(W_n)-y|\le\delta_n$ and 
 $g_r(W_n)\le\Phi_{s,r}(g_s(W_n))+\delta_n$,
but
\(
\delta_\cut(W_n,R_s(\gamma_y))\ge\eps_0
\)
for every $n$. By Theorem~\ref{thm:main},
$ 0\le g_r(W_n)-\Phi_{s,r}(g_s(W_n))\le\delta_n$.
Hence
\(
 g_s(W_n)\to y,\) and
 \(g_r(W_n)-\Phi_{s,r}(g_s(W_n))\to0.
\)
By compactness of graphon space in cut distance, after passing to a subsequence we may assume that
\(
 W_n\to W
\)
in cut distance. Clique densities are continuous in the cut topology, and $\Phi_{s,r}$ is continuous. Hence
\(
 g_s(W)=y,\) and 
 \(g_r(W)=\Phi_{s,r}(y)=P_r(t,a).
\)
By Proposition~\ref{prop:equality-reduction},
\(
 g_2(W)=\gamma_y\) and \(
 W\in R_s(\gamma_y).
\)
Therefore $\delta_\cut(W_n,R_s(\gamma_y))\le\delta_\cut(W_n,W)\to0,$
contradicting the choice of $W_n$. This proves the theorem.
\end{proof}

\begin{proof}[Proof of Theorem~\ref{thm:finite-stability}]
Suppose the assertion fails. Then there exist $\eps_0>0$, a sequence $\delta_n\to0$, and graphs $G_n$ with $|V(G_n)|=n\to\infty$ such that
\(
 g_r(G_n)\le\Phi_{s,r}(g_s(G_n))+\delta_n,
\)
but $G_n$ cannot be changed in at most $\eps_0n^2$ adjacencies into any member of $H_{s,r,n}$.

Applying Theorem~\ref{thm:main} to the step graphon of $G_n$ gives
$ 0\le g_r(G_n)-\Phi_{s,r}(g_s(G_n))\le\delta_n.$
In particular,
\(
 g_r(G_n)-\Phi_{s,r}(g_s(G_n))\to0.
\)
Let $W_n$ be the step graphon of $G_n$. Passing to a subsequence, we may assume that
\(
 W_n\to W
\)
in cut distance. Put
\(
 x=g_s(W).
\)
By continuity of clique densities and of $\Phi_{s,r}$, and by the asymptotic equality above,
\(
 g_r(W)=\Phi_{s,r}(x).
\)
If
\(
 x\le\sigma_s(r-1),
\)
then $\Phi_{s,r}(x)=0$, and hence $g_r(W)=0$. It follows that
\(
 N_r(G_n)=o(n^r).
\)
By the graph removal lemma of Erd\H os, Frankl and R\"odl~\cite{ErdosFranklRodl1986}, we can make $G_n$ $K_r$-free by changing $o(n^2)$ adjacencies. The resulting graph is a member of $\mathcal F_{r,n}\subseteq H_{s,r,n}$, contradicting the assumed $\eps_0n^2$-separation.

Thus
\(
 x>\sigma_s(r-1).
\)
By Proposition~\ref{prop:equality-reduction}, $W$ is an equality case for Theorem~\ref{thm:reiher}. Therefore, with $\gamma=g_2(W)$,
\(
 g_s(W)=F_s(\gamma).
\)
Since $W_n\to W$ in cut distance and $F_s$ is continuous,
\(
 g_s(W_n)-F_s(g_2(W_n))\to0.
\)
By Theorem~\ref{thm:reiher}, this difference is nonnegative for every $n$. Thus the graphs $G_n$ are asymptotically extremal for Theorem~\ref{thm:reiher}. Let $\eta>0$, and let $\delta_\eta>0$ and $n_\eta$ be the constants supplied by Theorem~\ref{thm:KLP} with $m=s$ and $\eps=\eta$. Since
\(
 g_s(G_n)-F_s(g_2(G_n))\to0,
\)
all sufficiently large $n$ satisfy $n\ge n_\eta$ and
$ g_s(G_n)\le F_s(g_2(G_n))+\delta_\eta.$
Therefore $G_n$ can be changed by at most $\eta n^2$ adjacencies into a member of $H_{s,n}\subseteq H_{s,r,n}$. Since $\eta>0$ is arbitrary, the edit distance from $G_n$ to $H_{s,r,n}$ is $o(n^2)$, contradicting the assumed $\eps_0n^2$-separation. This proves the theorem.
\end{proof}

\section*{Acknowledgements}

Yixiao Zhang is very grateful for the kind hospitality of IBS ECOPRO during his visit. 
After writing this manuscript, we learnt that Jie Ma, Tianhen Wang
and Tianming Zhu~\cite{MaWangZhu2026} had independently and simultaneously obtained similar results. 
We thank them for their communication. 

\paragraph{Declaration on the use of generative AI.}

The proof strategy was developed by the authors, inspired by the variational-and-link philosophy of Kim, Liu, Pikhurko and Sharifzadeh~\cite{KimLiuPikhurkoSharifzadeh2020}. The authors used generative AI tools to help check and simplify routine technical computations in several auxiliary lemmas. 
All AI-assisted suggestions and computations were independently checked and verified by the authors.

\bibliographystyle{abbrv}
\bibliography{clique_density}

\end{document}